\documentclass[12pt,thmsa, reqno]{amsart}

\usepackage{pdfpages}

\input{amssym.def}
\input{amssym.tex}

\usepackage{epstopdf}

\usepackage{enumitem}
\usepackage{graphicx}

\usepackage{lineno}

\usepackage{multicol}
\usepackage{amscd,mathrsfs}
\usepackage{caption}
\usepackage{amsmath}

\def\sn{\bbs^n}

\def\bbr{{\Bbb R}}

\def\bbc{{\Bbb C}}

\def\bbd{{\Bbb D}}

\def\bbs{{\Bbb S}}

\def\min{{\hbox{\rm min}}}

\def\rn{\bbr^n}

\def\part{\partial}
\def\intl{\int\limits}

\def\Gam{\Gamma}
\def\Om{\Omega}
\def\a{\alpha}

\def\Del{\Delta}
\def\del{\delta}
\def\vp{\varphi}

\def\gam{\gamma}
\def\Lam{\Lambda}
\def\sig{\sigma}
\def\lam{\lambda}
\def\z{\zeta}
\def\th{\theta}
\def\e{\varepsilon}
\def\t{\tau}

\def\snm1{\bbs^{n-1}}

\newtheorem{theorem}{Theorem}[section]
\newtheorem{lemma}[theorem]{Lemma}

\theoremstyle{definition}

\theoremstyle{remark}
\newtheorem{remark}[theorem]{Remark}

\theoremstyle{corollary}
\newtheorem{corollary}[theorem]{Corollary}

\numberwithin{equation}{section}

\newcommand{\be}{\begin{equation}}
\newcommand{\ee}{\end{equation}}

\newcommand{\bea}{\begin{eqnarray}}
\newcommand{\eea}{\end{eqnarray}}
\newcommand{\Bea}{\begin{eqnarray*}}
\newcommand{\Eea}{\end{eqnarray*}}

\def\sideremark#1{\ifvmode\leavevmode\fi\vadjust{\vbox to0pt{\vss
 \hbox to 0pt{\hskip\hsize\hskip1em
\vbox{\hsize2cm\tiny\raggedright\pretolerance10000
 \noindent #1\hfill}\hss}\vbox to8pt{\vfil}\vss}}}%

\begin{document}

\title[One-sided Fractional Integrals ]
{One-sided Fractional Integrals and  Riesz Potentials on a Spherical Cap}

\author{ B. Rubin}

\address{Department of Mathematics, Louisiana State University, Baton Rouge,
Louisiana 70803, USA}
\email{borisr@lsu.edu}

\subjclass[2020] {31B10, 45P05, 47G40}



\maketitle

\begin{abstract} We introduce fractional integrals  on the $n$-dimensional spherical cap, study their boundednes in weighted $L^p$ spaces  and obtain explicit inversion formulas.
The results are applied to the inversion problem for  Riesz  potentials on a spherical cap. In the case $n=2$, this problem arises in electrostatics, elasticity, and some other applied areas. The main tools are stereographic projection,  hypersingular integrals of the Marchaud type,  Poisson integrals, and spherical harmonic decompositions.

 \end{abstract}

\section {Introduction}

Let $\Om$ be  an $n$-dimensional spherical cap, $n\ge 2$. The main objective of the present article is explicit inversion of the
 Riesz  potential
  \be\label {Riesz}
(I^\a_\Om\vp)(x)=c_{n,\a}\!\intl_{\Om}\!
|t\!-\!\t|^{\alpha-n}\vp(\t)\,d\t, \quad c_{n,\a}=\frac{\Gamma ((n\!-\!\alpha)/2)} {2^\a\pi^{n/2}\Gamma (\alpha/ 2)},\ee
where  $0<\a<n$,  $t \in \Om$. This problem arises in electrostatics, elasticity, and some other applied areas; see, e.g.,
Collins \cite{Coll, Coll2} (the axisymmetric case for $n=2$, $\a=1$) and  Homentcovschi  \cite{Hom} (the asymmetric case for $n=2$, $\a=1$).
 An explicit inversion formula in the asymmetric  case for  $n=2$, $0<\a<2$, was obtained by  Fabrikant,  Sankar,  Roytman, and Swamy   \cite{FSRS}; see also Shail \cite{Shail} and Bilogliadov \cite{Bilo18}     for the axisymmetric problem with any $n\ge 2$, $\a=n-1$.

 The main tool in all these works is  factorization of Riesz potentials in terms of the Riemann-Liouville (or Abel) fractional integrals and their modifications. In the simplest form, these integrals are defined by the formulas
 \be\label {ch2 tag 4.1}
(I^\a_{a +} \psi) (t)\! =\! \frac{1}{\Gamma
(\alpha)} \intl^t_a \!\frac{\psi (\t)\, d\t}
{(t - \t)^{1-\a}}, \quad
(I^\a_{b -} \psi) (t)\! = \!
\frac{1}{\Gamma (\alpha)} \intl_t^b \!\frac{\vp (\t)\, d\t}
{(\t - t)^{1 - \alpha}}; \ee
\[- \infty \le a < t < b \le \infty, \qquad \Gamma (\alpha) = \intl^\infty_0 t^{\alpha -1} e^{-t} dt.\]
The integrals $I^\alpha_{a +} \psi$ and $I^\alpha_{b -} \psi$ are called the
{\it left-sided} and {\it right-sided}, respectively.

The present article is an updated and  revised version of the 1988 preprint \cite{VolR} (joint with O. Volovlikova), which was alluded to in  \cite [p. 508]{Ru24},  but remained unpublished.
 We introduce  one-sided  fractional integrals on a spherical cap
(see (\ref{2.2left}), (\ref{2.2right})) and apply them  to inversion of the Riesz potential (\ref{Riesz})  for all $n\ge 2$, $\a\in (0, n)$.
These  integrals can be inverted by making use of hypersingular integrals of the Marchaud type. The theory of such hypersingular integrals  was developed in \cite {Ru24}, following the original work by A. Marchaud \cite{Marc}  and the book \cite{SKM}. This approach  differs from that in the afore-mentioned applied publications and stems from   the   results in \cite{Ru84, Ru94a} and \cite [Chapter 10] {Ru24} for  the  $n$-dimensional ball.
The transition from the $n$-sphere in $\bbr^{n+1}$  to $\rn$ is  performed  by making use of the  stereographic projection. One can also implement   spherical harmonic expansions to invert spherical fractional integrals in terms of the corresponding Fourier-Laplace series.

The consideration is carried out in  weighted $L^p$-spaces and might be of interest from the point of view of harmonic analysis. Section 2 contains preliminaries. In Section 3 we introduce  one-sided spherical fractional integrals, study their properties, and obtain  factorization formulas for the  Riesz potentials. Sections 4 and 5 are focused on the inversion formulas. Some auxiliary calculations are presented in Appendix.

\section  {Preliminaries}\label{Prel}

\subsection  {Notation}\label{Not}

As usual, $\bbr$ is the set of all real numbers, $\bbr_+=\{ a\in \bbr: a>0\}$;  $\bbr^{n+1}$  is the $(n+1)$-dimensional Euclidean space of points $x=(x_1, \ldots, x_{n+1})$; $S^n =\{ x\in \bbr^{n+1}: |x|=1\}$  is the unit sphere in $\bbr^{n+1}$ with the standard surface area measure $dx$;
\[
\langle a, b \rangle_S = \{x \in S^n: a < x_{n+1} < b \}, \qquad -1 \le a < b \le 1, \]
is the spherical segment  in $S^n$.

The coordinate hyperplane  $x_{n+1}=0$ will be identified with the $n$-dimensional Euclidean space $\rn$ of the points $\xi =(\xi_1, \ldots, \xi_n)$; $r=|\xi| =(\xi_1^2 + \cdots + \xi_n^2)^{1/2}$.
 The notation
\[
\langle a, b \rangle_B = \{\xi \in \rn: a < |\xi| < b \}, \qquad 0 \le a < b \le \infty; \]
will be used for the corresponding  ball layer. The letters $a,b$ may have a different meaning, depending on the context.

We set  $S^{n-1}$ to be  the  unit sphere
in $\bbr^{n}$, $\sigma_{n-1} =  2\pi^{n/2} \big/ \Gamma (n/2)$ is the surface area of $S^{n-1}$.  Given $\xi \in \rn \setminus \{0\}$, we denote $\xi'= \xi/|\xi| \in S^{n-1}$ and  write $f(\xi)\equiv f (r, \xi')$ in polar coordinates.

The Poisson integral of a function $\vp$ on $S^{n-1}$ is defined by the formula
\be\label {Pint} (\Pi \vp)(\xi', r) \!=\!\frac{1}{\sig_{n-1}}\! \intl_{S^{n-1}} \!\frac{1\!-\!r^2}{|r\xi' \!-\!\eta'|^n} \vp (\eta')  d\eta', \quad \xi' \in S^{n-1}, \; r\in [0,1).\ee
If $\{Y_{j,k} (\xi')\}$ is an orthonormal basis of spherical harmonics in $L^2 (S^{n-1})$,  $j\in
\{0, 1,2, \ldots \}$, $k \in \{1,2,\ldots d_n(j)\}$,  $ d_n(j)$  is the dimension of the subspace of spherical harmonics of degree $j$, then the corresponding Fourier-Laplace decomposition of $(\Pi \vp)(\xi', r)$ has the form
\be\label {Flap}
\!\!\!\!\!\! (\Pi \vp)(\xi', r) \sim  \!\sum\limits_{j=0}^\infty \sum\limits_{k=1}^{d_n(j)} r^j\vp_{j,k}Y_{j,k} (\xi'), \quad \vp_{j,k}=\!\intl_{S^{n-1}} \!\!\!\vp(\xi') Y_{j,k} (\xi') d\xi'; \ee
see, e.g., \cite[p. 523] {Ru24} and references therein.

We will be dealing with Riesz potentials $I^\a_\Om\vp$ of the form
(\ref{Riesz}), where   $\Om$ is  the spherical segment  $\langle a, b \rangle_S$ in $S^n$ or the ball layer $\langle a, b \rangle_B$ in $\rn$. The letters $t$ and $\t$ in (\ref{Riesz}) will be replaced by $x$ and $y$ (for $\sn$) and by $\xi$ and $\eta$ (for $\rn$).
If $\Om$ coincides with the entire space $\rn$ or  $\sn$, we  write $I^\a_R$ and $I^\a_S$ for the corresponding potentials  $I^\a_{\rn}$ and $I^\a_{\sn}$.

In the following, $I$ is the identity operator; $||\cdot ||_X$ denotes the norm in the space $X$;

\vskip 0.2truecm

\noindent $L^p (\Omega) = \big\{ f: \ \|f\|_{L^{ p} (\Omega)} =
\big( \ \intl_\Omega |f(x)|^p dx \big)^{1/p} < \infty \big\}, \quad 1
\le p < \infty$;

\vskip 0.2truecm

\noindent $L^\infty (\Omega) = \big\{ f: \ \|f\|_{L^\infty
(\Omega)} = \text {\rm ess} \sup\limits_{x \in \Omega} \ |f(x)| <
\infty \big\}$.

\vskip 0.2truecm

\noindent $p'$ is the conjugate exponent to $p: \ 1/p+1/p'=1$, $1\le p \le \infty$ (if $p=\infty$, we set $1/p'=0$).

\vskip 0.2truecm

\noindent $L^p (\Omega; w) = \{ f: \ \|f\|_{L^{p} (\Omega; w)}
= \| w f \|_{L^{p} (\Omega)} < \infty \}$.

\vskip 0.2truecm

The abbreviations
``$\lesssim$'' and ``$\simeq$'' will be used for ``$\le$'' and ``$=$'', respectively, if the
corresponding relation holds up to an inessential constant factor.
 Given a real-valued function $f$, we set $ f_{\pm}(x)= f(x)$ if  $\pm f(x)>0$ and  $ f_{\pm}(x)= 0$, otherwise.
 All integrals are  understood in the Lebesgue sense.

\subsection  {Fractional Integrals in the Space $L^p (\bbr_+, \rho)$}\label{RL fracs}

 The Riemann-Liouville   fractional integrals $I^\a_{a +} \psi$ and $I^\a_{b -} \psi$ on the interval $(a,b)$ are  defined by
(\ref {ch2 tag 4.1}).  We set $I^\a_{ +}= I^\a_{(-\infty) +}$, $I^\a_{-}= I^\a_{\infty -}$.

The following boundedness results for the operators  $I^\a_{0 +}$ (for $a=0$) and $I^\a_{-}$ on $\bbr_+$ were obtained in \cite{Ru86}.

\begin{theorem}\label {Th 1.1} Let $\rho (t)=(1+t)^\gam \,t^{\gam_1}$, $1<p<\infty$. If $\gam_1 <1/p'$, then $I^\a_{0 +}$ is bounded from $L^p(\bbr_+, \rho)$ to $L^p(\bbr_+, \rho_+)$, where
\[\rho_+ (t)=(1+t)^\del \, t^{\gam_1 -\a}, \quad \del = \frac{1}{p'}- \gam_1- \begin{cases}
\displaystyle{\frac{1}{p'}- \gam_1-\gam}, &\hbox {if} \quad \gam_1+ \gam < \frac{1}{p'}, \\
\displaystyle{\e}, \,\e >0, &\hbox {if} \quad \gam_1+ \gam \ge \frac{1}{p'}. \\
\end{cases}\]
\end{theorem}

\begin{theorem}\label {Th 1.2}  Let $\rho (t)=(1+t)^\gam \,t^{\gam_2}$, $1<p<\infty$. If $\gam +\gam_2 >\a-1/p$, then $I^\a_{-}$ is bounded from $L^p(\bbr_+, \rho)$ to $L^p(\bbr_+, \rho_-)$, where
\[\rho_- (t)=(1+t)^\del \, t^{\del_2}, \qquad \del= \gam +\gam_2 -\a-\del_2, \]
\[\del_2 =  \begin{cases}
\displaystyle{\gam_2-\a}, &\hbox {if} \quad \gam_2 > \a- \frac{1}{p}, \\
\displaystyle{\e -\frac{1}{p},\,  \e >0}, &\hbox {if} \quad \gam_2 \le \a- \frac{1}{p}. \\
\end{cases}\]
\end{theorem}

\subsection  {One-sided Ball Fractional Integrals and Factorization of Riesz Potentials on $\rn$}\label {OS ball}

Let $0\le a<b\le \infty$, $\xi \in \langle a, b\rangle_B \subset \rn$, $n\ge 2$.
 The integrals
\bea\label {ball tag 24.1}
&&(B_{a +}^{ \a} \vp)(\xi) = \frac{2}{\Gamma (\a)\, \sig_{n-1}} \intl_{\langle a, |\xi|
\rangle_B} \frac {(|\xi|^2 - |\eta|^2)^{ \a}}{|\xi-\eta|^n} \vp (\eta)\, d\eta, \\
\label {ball tag 24.2}
&&(B_{b -}^{ \a} \vp)(\xi) = \frac{2}{\Gamma (\a)\, \sig_{n-1}} \intl_{\langle |\xi|, b\rangle_B} \frac {(|\eta|^2 - |\xi|^2)^{ \a}}{|\xi-\eta|^n} \vp (\eta)\, d\eta, \eea
are called {\it the one-sided ball fractional integrals} of  order $\a >0$,
the left-sided and the right-sided, respectively   \cite{Ru94a, Ru24}.

Passing to polar coordinates, one can write these integrals in terms of the one-dimensional fractional integrals in the radial variable and the  Poisson integral in the spherical variable. Specifically,
\bea\label {ball tag 24.3}
&&(B^{ \a}_{a +} \vp) (\xi) = \frac{2}{\Gamma (\a)}
\intl^r_a (r^2 - \rho^2)^{{ \a} -1} \left( \frac{\rho}{r}
\right)^{n-2} \Pi [ \vp (\rho, \cdot) ] \big( \xi', \frac{\rho}{r}
\big) \rho\, d \rho, \qquad \qquad \\
\label {ball tag 24.4}
&&(B^{ \a}_{b-} \vp) (\xi) = \frac{2}{\Gamma (\a)}
\intl^b_r ( \rho^2 - r^2)^{{ \a} -1} \Pi [ \vp (\rho,\cdot) ] \big(\xi', \frac{r}{\rho} \big) \rho\, d \rho, \quad r = |\xi|. \qquad \qquad  \eea
In particular, if $\vp$ is a radial function, i.e., $\vp (\xi)=\vp_0 (r)$, then, owing to the equality $ \Pi [1]=1$ (see, e.g., \cite[p. 43, Th. 1.9(b)]{SW}, we have
$(B^{ \a}_{a +} \vp) (\xi) = (\tilde B^{ \a}_{a +} \vp_0) (r)$, $(B^{ \a}_{b-} \vp) (\xi) = (\tilde B^{ \a}_{b-} \vp_0) (r)$, where
\bea\label {ball tag 24.3x}
&&(\tilde B^{ \a}_{a +} \vp_0) (r) = \frac{2}{\Gamma (\a)}
\intl^r_a (r^2 - \rho^2)^{{ \a} -1} \left( \frac{\rho}{r}
\right)^{n-2}  \vp_0 (\rho) \rho\, d \rho, \qquad \qquad \\
\label {ball tag 24.4x}
&&(\tilde B^{ \a}_{b-} \vp_0) (r) = \frac{2}{\Gamma (\a)}
\intl^b_r ( \rho^2 - r^2)^{{ \a} -1} \vp_0 (\rho) \rho\, d \rho; \qquad \qquad  \eea
cf. (\ref{ch2 tag 4.1}).
The operators
\[B_+^\a = B_{0+}^\a, \qquad B_-^\a= B_{\infty -}^\a\]
will play an especially important role in our consideration because many results for $B^{ \a}_{a +}$ with $a\neq 0$ and $B^{ \a}_{b-}$ with $b\neq \infty$ can be derived from those for $B_{\pm}^\a$   using restriction and extension by zero.

\begin{lemma}\label {lemma 1.1} Let $\rho (\xi)=|\xi|^{\gam} (1+ |\xi|)^{\del}$,  $1<p<\infty$.
\begin{enumerate}[label= \rm {(\roman*)}]

\item  If $\gam<n/p'$, then the operator $B_+^\a$ is bounded from $L^p (\rn, \rho)$ to $L^p (\rn, \rho_+)$, where
\be\label {1.8}  \rho_+ (\xi)=|\xi|^{\gam_+} (1+ |\xi|)^{\del_+},\ee
\be\label {1.9} \gam_+= \gam -2\a, \quad \del_+ = \begin{cases}
\displaystyle{\del}, &\hbox {if} \quad \gam +\del< n/p', \\
\displaystyle{\frac{n}{p'} -\gam -\e,\,  \e >0}, &\hbox {if} \quad \gam +\del\ge n/p'. \\
\end{cases}\ee

\item  If $\gam +\del >2\a -n/p$,  then the operator $B_-^\a$ is bounded from $L^p (\rn, \rho)$ to $L^p (\rn, \rho_-)$, where
\be\label {1.10}  \rho_- (\xi)=|\xi|^{\gam_-} (1+ |\xi|)^{\del-},\ee
\bea
&&\!\!\!\!\!\!\!\displaystyle{\gam_-= \gam -2\a, \quad \del_- =\del,} \quad
 \hbox {if} \;\, \gam > 2\a \!- \!n/p, \nonumber\\
&&\!\!\!\!\!\!\!\displaystyle{\gam_-= -\frac{n}{p} +\e,\quad \del_- = \frac{n}{p} +\gam +\del -2\a+\e, \, \e>0}, \quad \hbox {if} \; \, \gam \le 2\a \!- \!n/p. \qquad \nonumber\eea
\end{enumerate}
\end{lemma}

\begin{proof} (i) By the contraction property of the Poisson integral,
\be\label {1.12} ||(B_+^\a \vp)(\sqrt {t}, \cdot)||_{L^p (S^{n-1})}\le t^{1-n/2} (I_{0+}^\a \psi_1)(t), \ee
where
$\psi_1(t)=t^{n/2 -1} ||\vp(\sqrt {t}, \cdot)||_{L^p (S^{n-1})}$ and
\be\label {1.13}
||t^{\gam_1} (1+t)^{\del/2} \psi_1(t)||_{L^p (\bbr_+)} \lesssim ||\vp||_{L^p (\rn, \rho)}, \qquad \gam_1=\frac{2-n}{2p'}+\frac{\gam}{2}.\ee
Indeed,
\bea
&&||t^{\gam_1} (1+t)^{\del/2} \psi_1(t)||^p_{L^p (\bbr_+)}\!= \!\intl_0^\infty (1+t)^{\del p/2}  t^{(n/2 -1+\gam_1)p}\, dt \!\intl_{S^{n-1}} \!|\vp(\sqrt {t}, \xi')|^p \,d\xi'\nonumber\\
&& \lesssim \intl_{\rn} |\xi|^{\gam p} (1+|\xi|)^{\del p} |\vp(\xi)|^p \, d\xi=||\vp||_{L^p (\rn, \rho)}. \nonumber\eea
Further, by (\ref {1.12}) and Theorem \ref{Th 1.1},
\bea
&&||B_+^\a \vp||^p_{L^p (\rn, \rho_+)}=\intl_0^\infty r^{n-1}  \rho_+^p (r) \,dr \intl_{S^{n-1}}  |(B_+^\a \vp)(r, \xi')|^p d\xi'\nonumber\\
&&=\frac{1}{2}\intl_0^\infty   t^{n/2 -1}\rho_+^p (\sqrt {t})\, ||(B_+^\a \vp)(\sqrt {t}, \cdot)||^p_{L^p (S^{n-1})} dt \nonumber\\
&&\lesssim \intl_0^\infty   t^{p\gam_+/2+(1-n/2)(p-1)} (1+t)^{p\del_+/2} |(I_{0+}^\a \psi_1)(t)|^p dt \nonumber\\
&& \le \intl_0^\infty   t^{\gam_1 p} (1+t)^{\del p/2}  |\psi_1(t)|^p dt  \lesssim   ||\vp||^p_{L^p (\rn, \rho)}, \nonumber\eea
as desired.

(ii) As in (\ref {1.12}), we obtain
\be\label {1.14} ||(B_-^\a \vp)(\sqrt {t}, \cdot)||_{L^p (S^{n-1})}\le  (I_{-}^\a \psi_2)(t), \ee
where
$\psi_2(t)= ||\vp(\sqrt {t}, \cdot)||_{L^p (S^{n-1})}$ and
\be\label {1.15}
||t^{\gam_2} (1+t)^{\del/2} \psi_2(t)||_{L^p (\bbr_+)} \lesssim ||\vp||_{L^p (\rn, \rho)}, \qquad \gam_2=\frac{n-2}{2p}+\frac{\gam}{2}.\ee
Owing to the same reasoning as in (i), these inequalities, combined with Theorem \ref{Th 1.2}, yield the desired statement for $B_-^\a$.
\end{proof}

Explicit inversion formulas for the  fractional integrals $B_{\pm}^\a$ can be obtained by making use of hypersingular integrals of the Marchaud type with finite differences. Assuming $\ell >\a$, we denote

\be\label {findiff} (\Del^\ell_\t f )(u)= \sum\limits_{j=0}^{\ell} \binom{\ell}{j} (-1)^j f(u\!-\!j\t);  \quad \varkappa (\a, \ell)\!=\!\intl_0^\infty\frac{(1\!-\!e^{-\t})^\ell}{\t^{1+\a}}d\t,    \ee
\[ (A_\xi^{+} F)(u)=\begin{cases}
\displaystyle{ \Big ( 1+ \frac{u}{|\xi|^2} \Big )_+^{n/2 -1}}   &\displaystyle{\Pi[F(\sqrt{|\xi|^2 +u}, \cdot)] \left ( \Big ( 1+ \frac{u}{|\xi|^2} \Big )^{1/2}, \xi'\right ),}\\
 &\hbox {if} \quad -|\xi|^2<u<0, \\
\displaystyle{F(\xi)},  &\hbox {if} \quad u=0, \\
\end{cases}\]
\[ (A_\xi^- F)(u)=\begin{cases}
\displaystyle{\Pi[F(\sqrt{|\xi|^2 -u}, \cdot)] \left ( \Big ( 1- \frac{u}{|\xi|^2} \Big )^{-1/2}, \xi'\right ),}&\hbox {if} \quad u<0, \\
\displaystyle{F(\xi)},  &\hbox {if} \quad u=0, \\
\end{cases}\]
where $\Pi$ stands for the Poisson integral  (\ref{Pint}), $\xi\in \rn \setminus \{0\}$, $\xi' =\xi/|\xi|$. Then  $(B_\pm^\a )^{-1} F=\bbd^{ \a}_{\pm} F$, where
\be\label {3.2}
(\bbd^{ \a}_{\pm} F)(\xi)=\frac{1}{\varkappa (\a, \ell)}\intl_0^\infty\frac{(\Del^\ell_\t A_\xi^\pm F )(0)}{\t^{1+\a}}\, d\t.\ee

To give these hypersingular integrals precise meaning, we consider the truncated integrals
\be\label {ball tag 25.10}
(\bbd^{ \a}_{\pm,  \e} F ) (\xi) = \frac{1}{\varkappa_{\ell} (\a)}
\intl^\infty_{ \e} \frac{(\Delta^\ell_{ \tau} A^\pm_\xi F)(0)}{\tau^{1 +  \a}} d \tau, \quad \e > 0. \ee

\begin{theorem}\label{balt} Let $g \in L^p (\rn; \rho)$, $\ 1 < p \le \infty$, $\rho (\xi)=|\xi|^{\lam} (1+ |\xi|)^{\del -\lam}$.

\noindent {\rm (a)} If $\lam < n/p'$, $-\infty < \del <\infty$, then $\bbd^{ \a}_{+} B_+^\a g\equiv \, \stackrel {a.e.}{\lim\limits_{\e \to 0}} \bbd^{ \a}_{+,  \e} B_+^\a g =g$.

\noindent {\rm (b)} If  $ \del > 2\a -n/p$, $-\infty < \lam <\infty$, then $\bbd^{ \a}_{-} B_+^\a g\equiv \,\stackrel {a.e.}{\lim\limits_{\e \to 0}} \bbd^{ \a}_{-,  \e} B_-^\a g =g$.

\end{theorem}

These statements are particular cases of the more general  Theorems 10.23 and 10.25 in \cite{Ru24}. If $1 < p <\infty$, then, the limit in these formulas can also be understood in  the $L^p (\rn; \rho)$-norm under some additional assumptions for $\lam$ and $\del$. To avoid technicalities, we restrict our consideration to the a.e. convergence, when the simpler assumptions are sufficient owing to the one-sided structure of our integrals.

In the case $0<\a<1$, the  formula (\ref{3.2}) yields the following elegant expressions; see \cite[Section 10.3.3]{Ru24}:
\bea\label {ball tag 25.6}
 (\bbd^{ \a}_{+}F) (\xi) &=& \frac{\Gamma (n/2)}{\Gamma
(n/2 -\a)} \ \frac{F (\xi)}{|\xi|^{2 { \a}}}  \\
&+& \frac{2 \a}{\sig_{n-1} \Gamma (1 - \a)} \int\limits_{|\eta| < |\xi|}
\frac{F (\xi) - F (\eta)}{(|\xi|^2 - |\eta|^2)^{ \a}\, |\xi - \eta|^n}\, d\eta.\nonumber\eea
\be\label {ball tag 25.9}
(\bbd^{ \a}_{-}F) (\xi) = \frac{2 \a}{\sig_{n-1} \Gamma
(1 - \a)} \int\limits_{|\eta| > |\xi|} \frac{F (\xi) - F (\eta)}{(|\eta|^2 -
|\xi|^2)^{ \a}\, |\xi -\eta|^n} \,d\eta. \ee

The following factorization of Riesz potentials  on $\rn$  plays a crucial role in our consideration.
\begin{theorem}\label{ball  Theorem 24.4} Let $\a \in (0, n)$, $n \ge 2$. If the integral $I_\bbr^{ \a} \vp$ exists in the Lebesgue sense, then
\be\label {ball tag 24.19}
I_\bbr^{ \a} \vp = 2^{- \a} B_+^{{ \a}/2} \tau_{- \a}
B_-^{{ \a}/2} \vp = 2^{- \a} B_-^{{ \a}/2} \tau_{- \a}
B_+^{{ \a}/2} \vp, \ee
where  $\tau_{-\a}=|\xi|^{-\a}I$ stands for the multiplication operator.
\end{theorem}
This statement follows by Fubini's theorem from the equality
\[
 c_{n,\a}\,|\xi - \eta|^{\a - n}  = \frac{2^{2 - { \a}}}
{\Gamma^2 (\a/2) \, \sig^2_{n-1}} \int\limits_{\Om^\pm_{{\xi},
{\eta}}} \frac{||\xi|^2 -|\z|^2 |^{{ \a} / 2} \ ||\eta|^2 - |\z|^2 |^{{ \a}/2}} {|\xi - \z|^n \ |y
- \z|^n \ |\z|^{ \a}}  d\z; \]
\[\Om^+_{\xi,\eta} = \langle 0, \min
(|\xi|; |\eta|) \rangle_B, \qquad \ \Om^-_{\xi,\eta} = \langle \max ( |\xi|; |\eta| ) , \infty
\rangle_B;\]
see \cite[Lemma 10.4, Theorem 10.5]{Ru24} or \cite[Theorem A]{Ru94a} for details.

\subsection  {Stereographic Projection}\label {Stereo}

The main results for  fractional integrals on the unit sphere will be derived from those on the Euclidean ball. Below we prepare the background for this transition. We regard the unit sphere $\sn$ as a set of the form
\be\label {1.17}  S^n =((0,\pi) \times S^{n-1}) \cup \{e_{n+1}\} \cup \{-e_{n+1}\},\ee
where $\pm e_{n+1}=(0, \ldots, 0, \pm 1)$.  Then every point $x \in \sn \setminus \{\pm e_{n+ 1}\}$  can be identified with the pair
\be\label{pola}  x=(\theta , \xi'), \qquad 0<\th <\pi, \quad \xi' \in S^{n-1}. \ee
Similarly we set $\rn=(\bbr_+ \times S^{n-1}) \cup \{0\}$, so that for every  $\xi\in \rn\setminus \{0\}$ we write
$\xi=(r, \xi')$ , $r= |\xi|$. Let $\dot {\bbr}^n$  be the one-point compactification of $\rn$  homeomorphic to $S^n$.
Consider the stereographic projection $\gam : \sn \to \dot {\bbr}^n$ defined by
\be\label{1.18}
x=(\theta , \xi') \stackrel {\gam}{\longrightarrow} \xi = (r, \xi'), \quad r=|\xi| =\tan \frac {\theta}{2}, \quad \text {\rm if}\quad x\neq \pm e_{n+ 1},\ee
and $-e_{n+1} \stackrel {\gam}{\longrightarrow} 0$, $e_{n+1} \stackrel {\gam}{\longrightarrow} \infty$, otherwise; see Figure 1.

\begin{figure}[h]
\centering
\includegraphics[scale=.5]{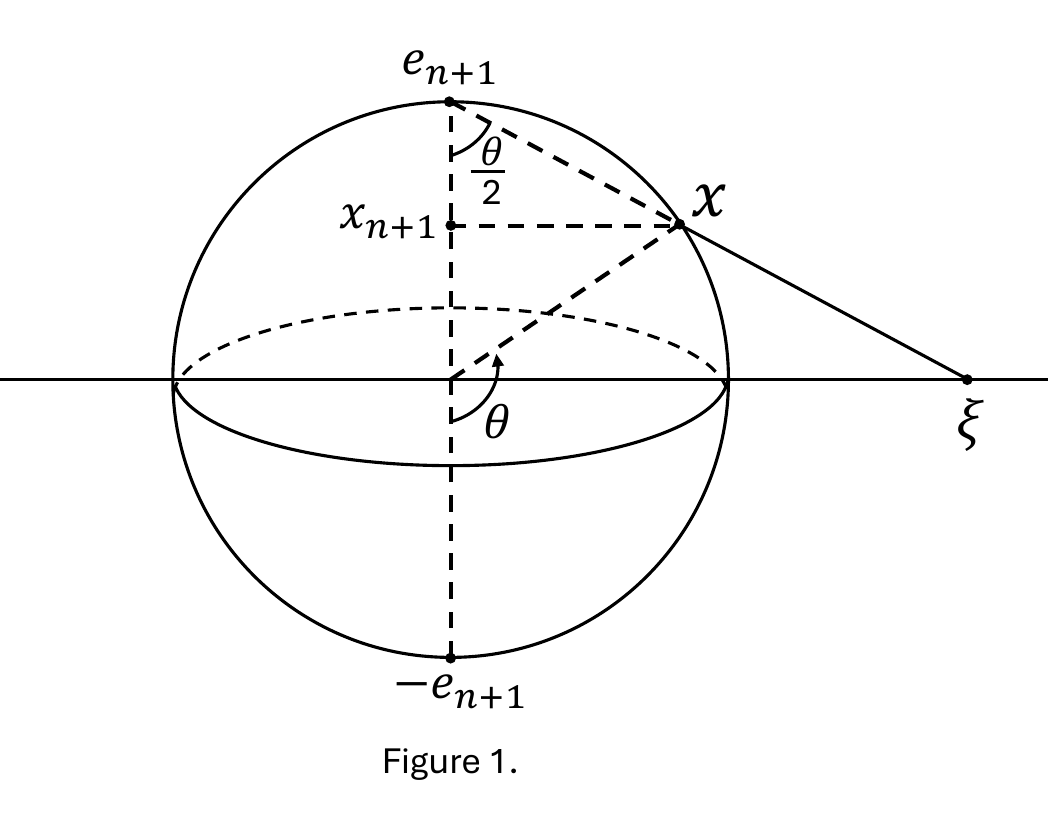}
\end{figure}

For  functions $f:\sn \to \bbc$ and $g:\rn \to \bbc$, we denote
\be\label{ster proj}  (Pf)(\xi)= f(\gam^{-1} \xi), \qquad (P^{-1}g)(x)=g (\gam x),\ee
where $\xi\in \rn$, $x\in \sn$. If $\xi \neq 0, \infty$, then
\be\label{1.19} (Pf)(\xi)= f(2 \tan^{-1} |\xi|, \xi').\ee
If $x\neq \pm e_{n+1}$, then
\be\label{1.20}
(P^{-1}g)(x)=  (P^{-1}g)(\th, \xi')= g\left (\tan \frac{\th}{2}, \xi'\right ).\ee

The following relations hold (cf. \cite[pp. 35-36]{Mi2}):
\be\label{1.21} |x-y|=\frac{2|\xi -\eta|}{(1+ |\xi|^2)^{1/2}\, (1+ |\eta|^2)^{1/2}}; \qquad \xi=\gam x, \; \eta =\gam y; \ee
\be\label{1.22}  1+\tan^2 \frac{\th}{2}=1+ |\xi|^2= \frac{2}{1-x_{n+1}}; \ee
\be\label{1.23}  1-x_{n+1}=\frac{2}{1+ |\xi|^2}, \qquad  1+x_{n+1}=\frac{2|\xi|^2}{1+ |\xi|^2}; \ee
\be\label{1.24}  \tan^2 \frac{\th}{2}= |\xi|^2=\frac{1+x_{n+1}}{1-x_{n+1}}; \ee
\be\label{1.25}  dx= \frac{2^n\,d\xi}{(1+ |\xi|^2)^n}; \qquad d\xi=\frac{dx}{(1-x_{n+1})^n}. \ee
We will  be using the following notations for the multiplication operators:
\bea
&&t_\lam=(1-x_{n+1})^\lam I; \qquad q_\lam=(1-x^2_{n+1})^{\lam/2} I; \nonumber\\
&& r_\lam=(1+ |\xi|^2)^{\lam/2} I; \qquad \t_\lam=|\xi|^\lam I.\nonumber\eea

\section {One-sided Fractional Integrals and Factorization of Riesz Potentials on the Sphere} \label {Sec 2}

In this section we  introduce one-sided spherical fractional integrals and investigate their properties in  weighted $L^p$-spaces. By making use of the stereographic projection, we establish connection
between these integrals and the  ball  integrals (\ref {ball tag 24.1}), (\ref {ball tag 24.2}). A factorization formula  for the spherical Riesz potentials $I^\a_S\vp$ in terms of the one-sided spherical integrals is obtained. This formula will play a key role  in Section \ref{inrep} for inversion of Riesz potentials on spherical caps.

\subsection {Definitions and Connections}

Given a point $x=(x_1, \ldots, x_{n+1})$  in $\sn$, consider the spherical segments
\bea
\langle a, x_{n+1} \rangle_S &=& \{y \in S^n: a < y_{n+1} < x_{n+1}\},  \nonumber\\
 \langle x_{n+1}, b \rangle_S &=& \{y \in S^n: x_{n+1} < y_{n+1} < b \},\nonumber\eea
where $-1 \le a < b \le 1$. 
We define the left-sided and right-sided  spherical fractional integrals of order $\a>0$ by the formulas
\be\label {2.2left}
(S^\a_{a+} \vp)(x)=\frac{2}{\Gam (\a)\, \sig_{n-1}} \intl_{\langle a, x_{n+1} \rangle_S} \frac{(x_{n+1} -y_{n+1})^\a}{|x-y|^n} \vp (y)\, dy,\ee
\be\label {2.2right}
(S^\a_{b-} \vp)(x)=\frac{2}{\Gam (\a) \,\sig_{n-1}} \intl_{\langle x_{n+1}, b \rangle_S} \frac{(y_{n+1} -x_{n+1})^\a}{|x-y|^n} \vp (y)\, dy,\ee
and denote $S^\a_{+}=S^a_{-1+}$, $ S^\a_{-}=S^a_{1-}$. Clearly,
\be\label {2.6}
 S^\a_{\pm} \vp=A S^\a_{\mp} A\vp, \quad (A\vp)(x)= \vp(x_1, \ldots, x_n, -x_{n+1}).\ee

The stereographic projection (\ref {ster proj}) establishes a remarkable connection between these integrals and the corresponding ball fractional integrals (\ref {ball tag 24.1}), (\ref {ball tag 24.2}).

\begin{lemma}\label{Lemma 2.2} Let $t_\lam =(1-x_{n+1})^\lam I$. We have
\bea\label {2.3} B^\a_{\pm} \psi&=&2^\a \,P\, t_{n/2 -\a} S^\a_{\pm}\, t_{-n/2 -\a}P^{-1} \psi, \\
\label {2.4}  S^\a_{\pm} \vp&=&2^{-\a}\, t_{\a-n/2} P^{-1} B^\a_{\pm} P\, t_{n/2 +\a} \vp,\eea
provided that either side of the corresponding  equality exists in the Lebesgue sense.
\end{lemma}
\begin{proof} For the integral $B^\a_{+} \psi$, owing to (\ref{1.21}), (\ref{1.23}), (\ref{1.25}), we have
\bea   &&(P^{-1} B^\a_{+} \psi)(x)=\frac{2^{1+\a}}{\Gam (\a)\, \sig_{n-1}} \nonumber\\
&&\times \intl_{\langle -1, x_{n+1} \rangle_S}
\frac{\Big (\frac{1}{1-x_{n+1}}-\frac{1}{1-y_{n+1}}\Big )^\a}
{|x-y|^n} \, \Big (\frac{1-x_{n+1}}{1-y_{n+1}} \Big )^{n/2}  (P^{-1} \psi)(y) dy\nonumber\\
&&=\frac{2^{1+\a}}{\Gam (\a)\, \sig_{n-1}}  (1-x_{n+1})^{n/2 -\a} \intl_{\langle -1, x_{n+1} \rangle_S} \frac{(x_{n+1}- y_{n+1})^\a  (P^{-1} \psi)(y)}{|x-y|^n  (1-y_{n+1})^{n/2 +\a} } dy\nonumber\\
&&=2^\a \, t_{n/2 -\a} S^\a_{+}\, t_{-n/2 -\a}P^{-1} \psi.\nonumber\eea
This gives (\ref{2.3}) for $B^\a_{+} \psi$. The proof  for $B^\a_{-} \psi$ is similar. The equalities (\ref{2.4}) follow from (\ref{2.3}).
\end{proof}

The case of zonal (or axisymmetric) functions is of particular interest for applications. Passing to spherical polar coordinates  $x=(\theta , \xi')$, $0<\th <\pi$, $ \xi' \in S^{n-1}$, suppose that $\vp (x)$ depends on $\theta$ only, that is, $\vp$ is zonal. Then, by rotation invariance,  $S^\a_{\pm} \vp$ are zonal, too, and represented by one-dimensional fractional integrals as follows.

\begin{theorem}\label{zonal} Let $\vp (x)\equiv \vp(\theta , \xi')=\vp_0 (\tan \theta/2)$.  Then
\[ (S^\a_{\pm} \vp)(\theta , \xi')=(\tilde S^\a_{\pm} \vp_0) (\tan \theta/2),\]
 where
\bea\label{zonp}
(\tilde S^\a_{+} \vp_0) (r)&=&   c_\a (r) \!\intl_0^r \!(r^2 \!-\!\rho^2)^{\a -1}   \frac{\vp_0 (\rho)}{(1+\rho^2)^{\a +n/2}}\, \Big (\frac{\rho}{r}\Big )^{n-2} \rho\, d\rho, \\
\label{zonm}
(\tilde S^\a_{-} \vp_0) (r)&=&  c_\a (r) \intl_r^\infty (\rho^2 -r^2 )^{\a -1} \frac{\vp_0 (\rho)}{(1+\rho^2)^{\a +n/2}}\, \rho\, d\rho,     \eea
\be\label{tsal} c_\a (r)= \frac{2^{\a +1}}{\Gam (\a)} (1+r^2)^{n/2 -\a}, \qquad r= \tan \theta/2 \in (0, \infty).\ee
\end{theorem}
\begin{proof}  Starting from (\ref{2.4}), we denote  $\psi (\eta)=(P\, t_{n/2 +\a} \vp)(\eta)$, where $t_\lam =(1-x_{n+1})^\lam I$, $\eta \in \rn$. Then, by  (\ref{1.22})-(\ref{1.22}),
\[\psi (\eta)=2^{n/2 +\a}  (1+|\eta|)^2)^{ -n/2 -\a} \vp_0 (|\eta|),\] and  (\ref{ball tag 24.3x}) yields
\[
(B^\a_{+} \psi)(\xi)= \frac{2^{n/2 +\a+1}}{\Gamma (\a)}
\!\intl^r_0 (r^2\! -\! \rho^2)^{{ \a} -1} \left( \frac{\rho}{r}
\right)^{n-2} \!(1\!+\!\rho^2)^{-\a -n/2} \vp_0 (\rho)\, \rho\, d \rho,\]
where $r=|\xi|$. This gives (\ref{zonp}) if we make use of (\ref{2.4}) the second time. The proof of (\ref{zonm}) is similar.
\end{proof}

Generalizations of the formulas (\ref{zonp}) and (\ref{zonm}) for arbitrary, not necessarily zonal  functions $\vp (x)\equiv \vp (\theta, \xi')$ can be similarly obtained from (\ref{ball tag 24.3}) and
 (\ref{ball tag 24.4}) by making use of  the Fourier-Laplace decomposition (\ref{Flap}) of the Poisson integral in the $\xi'$-variable.
 \begin{theorem}\label{nonzonal}
 If $f^{\pm}=S^\a_{\pm} \vp$, and
\[
\vp (x)\equiv \vp (\th, \xi') \sim \sum\limits_{j,k} \vp_{j,k} (r) Y_{j,k} (\xi'), \qquad r=\tan \theta/2,\]
is the Fourier-Laplace decomposition of $\vp (\th, \xi')$ in the $\xi'$-variable, then the Fourier-Laplace decompositions of  $f^{\pm}$ have the form
\[
f^{\pm} (x)\equiv f^{\pm} (\th, \xi') \sim \sum\limits_{j,k} f^{\pm}_{j,k} (\tan \theta/2)\, Y_{j,k} (\xi'),  \]
where
\bea\label{genp}
f^{+}_{j,k} (r)&=&   c_\a (r) \!\intl_0^r \!(r^2 \!-\!\rho^2)^{\a -1}   \frac{\vp_{j,k} (\rho)}{(1+\rho^2)^{\a +n/2}}\, \Big (\frac{\rho}{r}\Big )^{j+n-2} \rho\, d\rho, \qquad \\
\label{genm}
f^{-}_{j,k} (r)&=&  c_\a (r) \intl_r^\infty (\rho^2 -r^2 )^{\a -1} \frac{\vp_{j,k}(\rho)}{(1+\rho^2)^{\a +n/2}} \,\Big (\frac{r}{\rho}\Big )^{j}\,\rho\, d\rho,     \eea
$c_\a (r)$ being defined by (\ref {tsal}).
\end{theorem}

Convergence of the series in this theorem is understood in the sense depending on the class of functions $\vp$, according to the standard   Fourier-Laplace theory on the sphere; see, e.g., \cite[Theorems A.15, A.46, A.87]{Ru15}.

\vskip 0.2 truecm

The next statement establishes connection between the Riesz potentials $I^\a_S$ and $I^\a_R$.

\begin{lemma}\label{Lemma 2.1} Let $P$ be the stereographic projection (\ref {ster proj}), $0<\a<n$. Then
\be\label{2.1}
I^\a_S\vp=2^\a P^{-1} r_{n-\a} I^\a_{\bbr} \,r_{-n-\a} P\vp, \qquad r_\lam= (1+|\xi|^2)^{\lam/2}, \ee
provided that either side of this equality exists in the Lebesgue sense.
\end{lemma}
\begin{proof}
Using the formulas (\ref {1.21}) and  (\ref {1.25}) with $\xi=\gam x$, $\eta =\gam y$, and setting
$\tilde \vp (\eta)=\vp (2\tan^{-1} |\eta|, \eta')=(P\vp)(\eta)$, we obtain
\bea
(I^\a_S\vp)(x)&=&c_{n,\a}\!\intl_{\sn}\!
|x\!-\!y|^{\alpha-n}\vp(y)\,dy\nonumber\\
&=& 2^\a c_{n,\a}  (1+|\xi|^2)^{(n-\a)/2}\!\intl_{\rn}\! \frac{\tilde \vp (\eta)\, d\eta}{|\xi-\eta|^{n-\a}\,  (1+|\eta|^2)^{(n+\a)/2}}, \nonumber\eea
which gives (\ref{2.1}).
\end{proof}

The factorization formula (\ref {ball tag 24.19}) for the Riesz potentials on $\rn$  yields a  similar result for  potentials on the sphere.

\begin{theorem}\label{Lemma 2.4} Let $\a \in (0, n)$, $n \ge 2$. If the integral $I_S^{ \a} \vp$ exists in the Lebesgue sense, then
\be\label {2.11}
I_S^{ \a} \vp = 2^{- \a} S_+^{\a/2} q_{- \a} S_-^{ \a/2} \vp = 2^{- \a} S_-^{\a/2} q_{- \a} S_+^{\a/2} \vp, \ee
where $q_{- \a} (x)= (1- x_{n+1}^2)^{- \a/2}I$  is the multiplication operator.
\end{theorem}
\begin{proof} We recall the notation for multiplication operators:
\bea
&& t_\lam= (1-x_{n+1})^\lam I; \qquad q_\lam= (1-x^2_{n+1})^{\lam/2} I; \nonumber\\
&& r_\lam= (1+ |\xi|^2)^{\lam/2} I  ; \qquad  \t_\lam = |\xi|^\lam I.\nonumber\eea
 By Lemma \ref {Lemma 2.1}, owing to (\ref{ball tag 24.19}) and (\ref{2.3}), we have
\bea
I^\a_S\vp&=&2^\a P^{-1} r_{n-\a} I^\a_{\bbr} \,r_{-n-\a} P\vp = P^{-1} r_{n-\a}B_+^{\a/2} \tau_{- \a}B_-^{\a/2} \,r_{-n-\a} P\vp\nonumber\\
&=& P^{-1} r_{n-\a} P t_{(n-\a)/2} S^{\a/2}_{+}\, t_{-(n+\a)/2}P^{-1} \tau_{- \a} P t_{(n-\a)/2} \nonumber\\
&\times&  S^{\a/2}_{-}\, t_{-(n+\a)/2}P^{-1}  \,r_{-n-\a} P\vp\nonumber\\
&=& S_+^{\a/2} q_{- \a} S_-^{ \a/2} \vp. \nonumber\eea
\end{proof}

\subsection {Mapping Properties}

The following theorem is an analogue of the Hardy-Littlewood theorem for fractional integrals.

\begin{theorem}\label {Th 2.1} If $1<p<n/a$, then the operators $S^\a_{\pm}$ are bounded from $L^p (\sn)$ to $L^{p_\a} (\sn)$, $p_\a=np/(n-\a p)$.
\end{theorem}
\begin{proof} Because $|S^\a_{\pm}\vp|\lesssim I^\a_S |\vp|$, the result follows from the similar statement for Riesz potentials with weighted $L^p$-densities on $\rn$  \cite{Vak}.\footnote{An idea of representation of the spherical Riesz potential by the similar potential on $\rn$
via the stereographic projection was communicated by  the author to S.G. Samko and used later by Pavlov and Samko \cite{PaS} and Vakulov \cite{Vak}; see also \cite[p. 153]{Sa5} for details.}
\end{proof}

\begin{theorem}\label{Th 2.2} Let $w(x)= (1+x_{n+1})^\mu (1-x_{n+1})^\nu$.

\begin{enumerate}[label= \rm {(\roman*)}]

\item If $\mu<n/2p'$, then $S^\a_{+}$ is bounded from $L^p (\sn; w)$ to $L^p (\sn; w_+)$, where  $w_+(x)= (1+x_{n+1})^{\mu_+} (1-x_{n+1})^{\nu_+}$,
 \be\label {2.8}
 \mu_+=\mu -\a, \quad \nu_+=\begin{cases}
\displaystyle{\nu -\a},  &\hbox {if} \; \nu -\a> -n/2p, \\
\displaystyle{-n/2p +\e}, \, \e >0, &\hbox {if} \; \nu -\a \le -n/2p. \\
\end{cases}\ee

 \item If $\nu<n/2p'$, then $S^\a_{-}$ is bounded from $L^p (\sn; w)$ to $L^p (\sn; w_-)$, where  $w_{-}(x)= (1+x_{n+1})^{\mu_-} (1-x_{n+1})^{\nu_-}$,
 \be\label {2.10}
 \nu_-=\nu -\a, \quad \mu_-=\begin{cases}
\displaystyle{\mu -\a},  &\hbox {if} \; \mu -\a> -n/2p, \\
\displaystyle{-n/2p +\e}, \, \e >0, &\hbox {if} \; \mu -\a \le -n/2p. \\
\end{cases}\ee
 \end{enumerate}
\end{theorem}
\begin{proof} (i) By Lemma \ref{Lemma 2.2}, $S^\a_{+} \vp=2^{-\a}\, P \,t_{\a-n/2} P^{-1} B^\a_{+} P\, t_{n/2 +\a} \vp$, and we need to check the action of each component in this equality.
\vskip 0.2 truecm

\noindent {\bf 1.}  $t_{n/2 +\a}=(1-x_{n+1})^{n/2 +\a}I:$
\[ L^p (\sn; w) \to L^p (\sn; w_1), \quad w_1(x)= (1+x_{n+1})^{\mu} (1-x_{n+1})^{\nu -n/2-\a}.\]

\vskip 0.2 truecm

\noindent {\bf 2.}  $P: L^p (\sn; w_1) \to L^p (\rn; w_2),$
\[ w_2 (\xi)=|\xi|^{\tilde \mu} (1+|\xi|)^{\tilde \nu}, \quad \tilde \mu= 2\mu, \quad \tilde \nu =-2\mu +n+2\a -2\nu -2n/p.\]
\vskip 0.2 truecm

\noindent {\bf 3.} By Lemma \ref {lemma 1.1}, if $\tilde \mu <n/p'$ , then $B^\a_+$ is bounded from $L^p (\rn; w_2)$ to $L^p (\rn; w_3)$, were
\[
w_3 (\xi)=|\xi|^{\mu_0} (1+|\xi|)^{\nu_0}, \quad \mu_0=\tilde \mu -2\a=2\mu -2\a,\]
\[
\nu_0= \begin{cases}
\displaystyle{-2\mu +n+2\a -2\nu -2n/p},  &\hbox {if} \; \a -\nu < n/2p, \\
\displaystyle{n/p' -2\mu -\e}, \, \e >0, &\hbox {if} \; \a -\nu \ge n/2p. \\
\end{cases}\]
\vskip 0.2 truecm

\noindent {\bf 4.} $P^{-1}: L^p (\rn; w_3) \to L^p (\sn; w_4)$,
\[ w_4(x)= (1+x_{n+1})^{\mu_1} (1-x_{n+1})^{\nu_1},  \quad \mu_1=\mu_0/2=\mu -\a,\]
\[
\nu_1= \begin{cases}
\displaystyle{\nu -n/2},  &\hbox {if} \; \a -\nu < n/2p, \\
\displaystyle{-n/p' -n +\a +\e}, \, \e >0, &\hbox {if} \; \a -\nu \ge n/2p. \\
\end{cases}\]

\vskip 0.2 truecm

\noindent {\bf 5.}   $t_{\a -n/2}=(1-x_{n+1})^{\a -n/2}I:   L^p (\sn; w_4) \to L^p (\sn; w_+)$,  where $w_+ (x)$ has the desired form, as stated in the theorem.
\vskip 0.2 truecm

(ii)   By Lemma (\ref{2.6}), $ S^\a_{-} \vp=A S^\a_{+} A\vp$, and we consider the action of each component of this formula, as above.

\vskip 0.2 truecm

\noindent {\bf 6.} The operator $A$  is bounded from $L^p (\sn; w)$ to $L^p (\sn; \bar w_1)$, where  $\bar w_1(x)= (1-x_{n+1})^{\mu} (1+x_{n+1})^{\nu}$,

\vskip 0.2 truecm

\noindent {\bf 7.}  Owing to the part (i), if $\nu<n/2p'$, then $S^\a_+$ is bounded from $L^p (\sn; \bar w_1)$ to $L^p (\sn; \bar w_2)$, where  $\bar w_2 (x)=(1-x_{n+1})^{\bar\nu_+} (1+x_{n+1})^{\bar\mu_+}$,
\[
\bar\mu_+=\nu -\a, \quad \bar\nu_+=\begin{cases}
\displaystyle{\mu -\a},  &\hbox {if} \; \mu -\a > - n/2p, \\
\displaystyle{- n/2p +\e}, \, \e >0, &\hbox {if} \; \mu -\a \le - n/2p. \\
\end{cases}\]

\vskip 0.2 truecm

\noindent {\bf 8.}  The operator $A$  is bounded from $L^p (\sn; \bar w_2)$ to $L^p (\sn; w_-)$,   where $w_- (x)$ has the desired form.
\end{proof}

\section {Inversion Formulas for  $S^\a_\pm \vp$}\label {sec 3}

By (\ref{2.4}),
\be\label {2.4x} S^\a_{\pm} \vp=2^{-\a}\, t_{\a-n/2} P^{-1} B^\a_{\pm} P\, t_{n/2 +\a} \vp,\ee
 $t_\lam =(1-x_{n+1})^\lam I$.   This gives (at least formally)
\be\label {2.4 inv}
 (S^\a_{\pm})^{-1}f=2^{\a}\, t_{-n/2-\a} P^{-1} (B^\a_{\pm})^{-1} P\,t_{-\a+n/2}f.\ee
It remains to  plug the expression of $ (B^\a_{\pm})^{-1}$ from (\ref{3.2}) in (\ref{2.4 inv}), which gives (see \cite{VolR} for detailed calculations)
\be\label {3.4} (S^\a_{\pm})^{-1}f (x)=
\frac{2^{-\a}}{\varkappa (\a, \ell)}(1+\tan^2(\th/2))^{\a +n/2} \intl_0^\infty\frac{(\Del^\ell_{\t} A_x^{\pm} f )(0)}{\t^{1+\a}}\, d\t.\ee
where (cf. (\ref{1.18}))
 \[x=(\theta, \xi'), \quad 0<\th <\pi, \quad \xi' \in S^{n-1}, \quad 1+\tan^2 (\th/2)=\frac{2}{1-x_{n+1}},\]
$\ell >\a$, $\Del^\ell_{\t}$ and $\varkappa (\a, \ell)$ are defined by (\ref{findiff}),
\bea
\label{4.1}(A_x^+ f)(u)&=& \Big(1+\frac{u}{\tan^2(\th/2)}\Big)_+^{n/2-1} (1+\tan^2(\th/2)+u)^{\a -n/2} \\
&\times&\Pi\left[f\Big(2 \tan^{-1} \sqrt{\tan^2(\th/2) +u}, \cdot\Big)\right] \Big ( \Big ( 1+ \frac{u}{\tan^2 (\th/2)} \Big )^{1/2}, \xi'\Big ), \nonumber\eea
\bea
\label{3.5}(A_x^- f)(u)&=& (1+\tan^2(\th/2)-u)^{\a -n/2} \\
&\times&\Pi\left[f\Big(2 \tan^{-1} \sqrt{\tan^2(\th/2) -u}, \cdot\Big)\right] \Big ( \Big ( 1- \frac{u}{\tan^2 (\th/2)} \Big )^{-1/2}, \xi'\Big ), \nonumber\eea
if $ u<0$, and
\be\label{3.5a}
(A_x^\pm f)(u)= (1+\tan^2(\th/2))^{\a -n/2} f(x), \qquad {\rm if}\quad u=0;\ee

 Our next aim is to give (\ref {3.4}) precise meaning. For $\e>0$, consider the truncated integrals
\be\label {3.6}
(T^\a_{\pm,\e} f) (x)\!=\!
\frac{2^{-\a}}{\varkappa (\a, \ell)}(1\!+\!\tan^2(\th/2))^{\a +n/2} \!\intl_0^\infty\frac{(\Del^\ell_{\t} A_x^{\pm} f )(0)}{\t^{1+\a}}\, d\t\ee
and set $T^{\a}_{\pm} f= \stackrel {a.e.}{\lim\limits_{\e \to 0}} T^{ \a}_{\pm,  \e} f$; cf. (\ref {3.4}).
The next theorem guarantees the existence of this limit when $f=S_\pm^\a \vp$ for a suitable function $\vp$.

\begin{theorem}\label{sphere new} Let $\vp\in L^p (\sn; w)$, where  $w(x)= (1+x_{n+1})^\mu  (1-x_{n+1})^\nu$, $ 1 < p \le \infty$.

\noindent {\rm (a)} If $\mu < n/2p'$, $-\infty < \nu <\infty$, then $T^{\a}_{+} S_+^\a \vp\equiv \, \stackrel {a.e.}{\lim\limits_{\e \to 0}} T^{ \a}_{+,  \e} S_+^\a \vp =\vp$.

\noindent {\rm (b)} If  $\nu < n/2p'$, $-\infty < \mu<\infty$, then $T^{\a}_{-} S_-^\a \vp\equiv \, \stackrel {a.e.}{\lim\limits_{\e \to 0}} T^{ \a}_{-,  \e} S_-^\a \vp =\vp$.

\end{theorem}
\begin{proof}   Let
$g(\xi) = (P\,t_{n/2 +\a}\vp)(\xi)$, $\xi \in \rn$; cf. (\ref{2.4x}). Denote
\be\label {3.10}
w_0 (\xi) = |\xi|^{\tilde \mu } (1+|\xi|)^{\tilde \nu }, \qquad \tilde \mu =2\mu, \quad \tilde \nu= -2\mu -2\nu +2\a +n-\frac{2n}{p}.\ee
Then  $||g||_{L^p (\rn; w_0)}  =  ||\vp||_{L^p (\sn; w)}$. Indeed, by (\ref{1.25}),
\bea
||g||^p_{L^p (\rn; w_0)}  &=&  ||P\,t_{n/2 +\a}\vp||^p_{L^p (\rn; w_0)}\nonumber\\
&=&2^{(n+2\a)p}\intl_{\rn} |\xi|^{2\mu p} (1+|\xi|)^{-2\mu p-2\nu p -2n} |(P\vp)(\xi)|^p \, d\xi\nonumber\\
&\simeq& \intl_{\sn} (1+x_{n+1})^{\mu p}   (1-x_{n+1})^{\nu p} |\vp (x)|^p\, dx= ||\vp||^p_{L^p (\sn; w)},\nonumber\eea
as desired. Now the result follows from (\ref{2.4 inv}) and Theorem \ref {balt}, if the latter is applied with $\lam = \tilde \mu$ and $\del -\lam = \tilde \nu$.
\end{proof}

In the case $0<\a<1$, the corresponding inversion  formulas, which follow from  (\ref{ball tag 25.6}) and  (\ref{ball tag 25.9}), look much simpler:
\bea\label {3.37}
(T^{\a}_{\pm} f) (x) &=& \frac{\Gam (n/2)}{\Gamma (n/2 - \a)\, (1\pm x_{n+1})^\a}\, f(x)\\
&+& \frac{2^\a}{\sig_{n-1}  \Gamma (1 - \a)}\int\limits_{\sn}
\frac{f(x) - f(y)}{ (x_{n+1}-y_{n+1})^\a_{\pm} \,  |x-y|^n}\, dy.\nonumber\eea

The proof of these formulas is given in Appendix.

For the spherical caps which differ from the entire sphere, the following corollaries hold, owing to the one-sided structure of the corresponding integrals.
\begin{corollary}\label {cor 3.1}
Let $a \in (-1,1)$, $\vp\in L^p (\langle a,1\rangle_S; (1- x_{n+1})^\nu)$, $\nu < n/2p'$,  $1<p\le \infty$.
Then $T^{\a}_{-} S_-^\a \vp\equiv \, \stackrel {a.e.}{\lim\limits_{\e \to 0}} T^{ \a}_{-,  \e} S_-^\a \vp =\vp$.
\end{corollary}

 Consider the spherical fractional integral $S_{b-}^\a \vp$, $b \in (-1,1)$, on the cap $\langle -1, b\rangle_S$ centered at the South Pole.
We denote by $P_{b+}$ and $P_{b-}$ the restriction operators onto the spherical caps  $\langle b, 1 \rangle_S $ and  $\langle -1, b \rangle_S $, respectively.  The notation  $\Lam_{b-}$ and $\Lam_{b+}$ are used for the corresponding extension by zero operators to $x_{n+1} < b$ and $x_{n+1} >b$. Let
\[T^{\a}_{b-, \e} f = P_{b-}T^{\a}_{-, \e} \Lam_{b+} f.\]

\begin{corollary}\label {cor 3.2} Suppose that $\vp\in L^p (\langle -1,b\rangle_S; (1+ x_{n+1})^\mu)$, $\mu < n/2p'$,  $1<p\le \infty$. Then
$T^{\a}_{b-} S_{b-}^\a \vp\equiv \, \stackrel {a.e.}{\lim\limits_{\e \to 0}} T^{ \a}_{b -,  \e} S_{b-}^\a \vp =\vp$.
\end{corollary}

Similar corollaries hold for the left-sided  fractional integrals.

\begin{corollary}\label {cor 4.1}
Let $b \in (-1,1)$, $\vp\in L^p (\langle -1, b\rangle_S; (1+ x_{n+1})^\mu)$, $\mu < n/2p'$,  $1<p\le \infty$.
Then $T^{\a}_{+} S_+^\a \vp\equiv \, \stackrel {a.e.}{\lim\limits_{\e \to 0}} T^{ \a}_{+,  \e} S_+^\a \vp =\vp$.
\end{corollary}

Consider the left-sided integral $S_{a+}^\a \vp$ on the upper cap $\langle a,1\rangle_S$. Let \[ T^{\a}_{a+, \e} f = P_{a+}T^{\a}_{+, \e} \Lam_{a-} f,\] where  $T^{\a}_{+, \e}$ is defined by (\ref{3.6}), $P_{a+}$ and $\Lam_{a-}$ stand for the corresponding restriction and  extension by zero operators.

\begin{corollary}\label {cor 4.2} Suppose that $\vp\in L^p (\langle a,1\rangle_S; (1- x_{n+1})^\nu)$, $\nu < n/2p'$,  $1<p\le \infty$.  Then
$T^{\a}_{a+} S_{a+}^\a \vp\equiv \, \stackrel {a.e.}{\lim\limits_{\e \to 0}} T^{ \a}_{a+,  \e} S_{a+}^\a \vp =\vp$.
\end{corollary}

\section{Inversion of Riesz Potentials}\label {inrep}

In this section we consider Riesz potentials  $I^\a_\Om \vp$ (see (\ref{Riesz})) over the spherical caps $\langle a,1\rangle_S$ and  $\langle -1,b,\rangle_S$. Our aim is to obtain  inversion formulas for   these potentials in the corresponding weighted $L^p$- spaces. The main tool is the following factorization theorem.

\begin{theorem}\label{th 5.1} Let $\a \in (0, n)$, $n \ge 2$. If the integral $I^\a_\Om \vp$ exists in the Lebesgue sense then it can be factorized by the formulas
\be\label {factor} I^\a_\Om \vp=\begin{cases}
\displaystyle{S_{+}^{\a/2} q_{-\a}   S_{b-}^{\a/2}\vp,}  &\hbox {if} \quad \Om=\langle -1,b\rangle_S, \\
\displaystyle{S_{-}^{\a/2} q_{-\a}   S_{a+}^{\a/2}\vp,}  &\hbox {if} \quad \Om=\langle a,1\rangle_S, \\
\end{cases}\ee
where $q_{-\a} =  (1- x_{n+1}^2)^{-\a/2}I$ is the multiplication operator.
\end{theorem}

This statement is an immediate consequence of Theorem \ref{Lemma 2.4}.

Because we already know how to invert each component in (\ref{factor}), we can now invert the potential on the left-hand side. Examination of conditions of Corollaries \ref{cor 4.1},  \ref{cor 4.2}, and Theorem
\ref{Th 2.2} yields the following result.

\begin{theorem}\label{th 5.2}  Let $f= I^\a_\Om \vp$, $0<\a<n$, $1<p < \infty$, $n\ge 2$.

\vskip 0.2 truecm

\noindent {\rm (i)} If $\Om=\langle -1,b\rangle_S$, $b\in (-1,1)$,   $\vp\in L^p (\Om; (1+ x_{n+1})^\mu)$, $\mu < n/2p'$,  then
\be\label {Inv cup}
\vp = T^{\a/2}_{b-} q_\a  T^{\a/2}_{+} f,\ee

\vskip 0.2 truecm
\noindent {\rm (ii)} If $\, \Om=\langle a,1\rangle_S$, $a\in (-1,1)$,   $\vp\in L^p (\Om; (1- x_{n+1})^\nu)$, $\nu < n/2p'$,  then
\be\label {Inv cap}
\vp = T^{\a/2}_{a+} q_\a  T^{\a/2}_{-} f.\ee
The convergence of the hypersingular integrals in these formulas is understood in the a.e. sense.
\end{theorem}

\begin{remark} The factorization Theorem \ref{th 5.1}, combined with the spherical harmonic expansions in Theorem \ref{nonzonal}, enables one to compute the Fourier-Laplace coefficients of the unknown density $\vp$ by solving two Abel-type integral equations corresponding to fractional integrals (\ref{genp}) and (\ref{genm}). Many inversion methods for such equations are discussed, e.g.,  in the books \cite{Ru15, Ru24, SKM}. We leave this exercise to the interested reader.
\end{remark}

\section{Appendix}

Recall that our aim is to prove the following inversion formulas:
\bea\label {Recall}
(T^{\a}_{\pm} f) (x) &=& \frac{\Gam (n/2)}{\Gamma (n/2 - \a)\, (1\pm x_{n+1})^\a}\, f(x)\\
&+& \frac{2^\a}{\sig_{n-1}  \Gamma (1 - \a)}\int\limits_{\sn}
\frac{f(x) - f(y)}{ (x_{n+1}-y_{n+1})^\a_{\pm} \,  |x-y|^n}\, dy,\nonumber\eea
where
\be\label {plug} T^{\a}_{\pm}f=(S^{\a}_{\pm})^{-1}f= 2^{\a}\, t_{-n/2-\a} P^{-1} (B^\a_{\pm})^{-1} P\,t_{-\a+n/2}f,\ee
$t_\lam =(1-x_{n+1})^\lam I$, $0<\a<1$.

Owing to (\ref{2.6}), it suffices to consider the operator $ (S^{\a}_{-})^{-1}$. Let us  plug the expression  (\ref{ball tag 25.9}) for $(B^\a_{-})^{-1}$ in (\ref{plug}). The transition formulas
(\ref{ster proj})-(\ref{1.25}) yield
\bea &&(S^{\a}_{-})^{-1}f)(x)= 2^{\a} (1-x_{n+1})^{-n/2 -\a}  P^{-1}\Bigg \{\frac{2\a}{\sig_{n-1} \Gam (1-\a)}\nonumber\\
&&\times\intl_{|\eta| >|\xi|} \frac{ \displaystyle{\Big (\frac{2}{1+|\xi|^2}\Big )^{-\a +n/2} (Pf)(\xi) -  \Big (\frac{2}{1+|\eta|^2}\Big )^{-\a +n/2} (Pf)(\eta)}}
{(|\eta|^2 -|\xi|^2)^\a \, |\xi -\eta|^n}\, d\eta\Bigg \}\nonumber\\
&&=\frac{2^\a}{\sig_{n-1} \Gam (1-\a)}
\intl_{\langle x_{n+1},1\rangle_S} \frac{\displaystyle{\Big (\frac{1-y_{n+1}}{1-x_{n+1}}\Big )^{\a -n/2} f(x) -f(y)}}
{(y_{n+1}-x_{n+1})^\a |x-y|^n }\,dy=I_1+I_2, \nonumber\eea
where
\bea &&I_1=\frac{2^\a\, f(x)}{\sig_{n-1} \Gam (1\!-\!\a)}
\!\intl_{\langle x_{n+1},1\rangle_S}\!\!\! \!\!\Big [\Big (\frac{1\!-\!y_{n+1}}{1\!-\!x_{n+1}}\Big )^{\a -n/2}\!\!\! -1\Big ]\, \frac{dy}{(y_{n+1}\!-\!x_{n+1})^\a |x\!-\!y|^n},\nonumber\\
&&I_2=\frac{2^\a}{\sig_{n-1} \Gam (1-\a)}
\intl_{\langle x_{n+1},1\rangle_S} \frac{ f(x) -f(y)}{(y_{n+1}-x_{n+1})^\a |x-y|^n }\,dy.\nonumber\eea

The expression $I_2$ gives the second term in (\ref{Recall}). To show that $I_1$ coincides with the first term, we denote
\[ \Lam= \frac{2^\a}{\sig_{n-1} \Gam (1-\a)}
\intl_{\langle x_{n+1},1\rangle_S}\!\!\! \!\!\Big [\Big (\frac{1\!-\!y_{n+1}}{1\!-\!x_{n+1}}\Big )^{\a -n/2}\!\!\! -1\Big ]\, \frac{dy}{(y_{n+1}\!-\!x_{n+1})^\a |x\!-\!y|^n}\]
and prove the following lemma.
\begin{lemma}
\be\label{3.33} \Lam=\frac{\Gam (n/2)}{\Gamma (n/2 - \a)\, (1- x_{n+1})^\a}.\ee
\end{lemma}
\begin{proof} Denote $\gam= 2^\a/\Gam (1-\a)$, $r=|\xi|$. We have
\bea
&&\Lam=\frac{\gam}{\sig_{n-1}}\intl_{|\eta| >|\xi|} \Big [\Big (\frac{1+|\xi|^2}{1-|\eta|^2}\Big )^{\a -n/2}\!\!\! -1\Big ]\nonumber\\
&&\times \frac{(1+|\xi|^2)^{n/2} (1+|\eta|^2)^{n/2}\, d\eta}{\displaystyle{ (1+|\eta|^2)^{n} \Big (\frac{2|\eta|^2}{1+|\eta|^2}- \frac{2|\xi|^2}{1+|\xi|^2}\Big )^\a  |\xi -\eta|^n}}\nonumber\\
&&= \frac{2^{-\a} \gam}{\sig_{n-1}}\intl_{|\eta| >|\xi|}\Big [\Big (\frac{1+|\xi|^2}{1-|\eta|^2}\Big )^{\a -n/2}\!\!\! -1\Big ]\frac{(1+|\xi|^2)^{n/2+\a}  d\eta}{ (1+|\eta|^2)^{n/2-\a}(|\eta|^2 -|\xi|^2)^\a |\xi -\eta|^n}\nonumber\\
&&= \frac{2^{-\a} \gam}{\sig_{n-1}}\intl_r^\infty \Big [\Big (\frac{1+r^2}{1-\rho^2}\Big )^{\a -n/2}\!\!\! -1\Big ]\frac{(1+r^2)^{n/2+\a} \rho^{-1} (1-r^2/\rho^2)^{-1} d\rho
}{ (1+\rho^2)^{n/2-\a}(\rho^2 -r^2)^\a }\nonumber\\
&&\times \intl_{S^{n-1}}\frac{1-r^2/\rho^2}{|\xi' r/\rho -\eta'|^n}\, d\eta'\nonumber\\
&&=\gam 2^{-\a}\intl_r^\infty  \Big [\Big (\frac{1+r^2}{1-\rho^2}\Big )^{\a -n/2}\!\!\! -1\Big ] \frac{(1+r^2)^{n/2+\a} \rho \, d\rho}{(1+\rho^2)^{n/2-\a}(\rho^2 -r^2)^{\a+1} }\nonumber\\
&&=2^{-\a}\, (1+r^2)^{n/2+\a}\,h_\a (r^2),\nonumber\eea
where
\bea &&h_\a (t)=\frac{\a}{\Gam (1-\a)}\intl_t^\infty \frac{(1+t)^{\a-n/2} - (1+\t)^{\a-n/2}}{(\t -t)^{\a+1}}\, d\t\nonumber\\
&&=\frac{\a}{\Gam (1-\a)}\intl_0^\infty \frac{(1+t)^{\a-n/2} - (1+t+\t)^{\a-n/2}}{\t^{\a+1}}\, d\t=(1+t)^{-n/2}\psi (\a),\nonumber\eea
\[
\psi (\a)=\frac{\a}{\Gam (1-\a)}\intl_0^\infty \frac{1-(1+s)^{\a-n/2}}{s^{\a+1}}\, ds.\]
Thus  (cf.   (\ref{1.23})),
\[\Lam=2^{-\a}\, (1+r^2)^{\a} \psi (\a)=\frac {\psi (\a)}{(1-x_{n+1})^\a}, \]
and it remains to show that
\be\label{ruak} \psi (\a)\equiv \frac{\a}{\Gam (1-\a)}\intl_0^\infty \frac{1-(1+s)^{\a-n/2}}{s^{\a+1}}\, ds=\frac{\Gam (n/2)}{\Gam (n/2-\a)}.\ee

 To this end, consider the integral
\[
J(\a)=\frac{1}{\Gam (\a)}\intl_0^\infty s^{\a-1} (1+s)^{-\a-n/2}\, ds= \frac{\Gam (n/2)}{\Gam (n/2+\a)}, \qquad Re \, \a >0,\]
and compute  analytic continuation  ($ a.c.$) of this equality to $Re \, \a \in (-1,0)$.   For $ Re \, \a >0$ we have
\bea
J(\a)&=&\frac{1}{\Gam (\a)}\Big (\intl_0^1 + \intl_1^\infty\Big )\, s^{\a-1} (1+s)^{-\a-n/2}\, ds\nonumber\\
&=&\frac{1}{\Gam (\a)}\intl_0^1 s^{\a-1}  \left [(1+s)^{-\a-n/2}-1\right ]\,ds \nonumber\\
&+& \frac{1}{\Gam (\a +1)}+ \frac{1}{\Gam (\a)} \intl_1^\infty s^{\a-1} (1+s)^{-\a-n/2}\, ds. \nonumber\eea
This gives
\bea \underset{Re \, \a \in (-1,0)}  { a.c.}
 J(\a)&=& \frac{1}{\Gam (\a)}\intl_0^1 s^{\a-1}  \left [(1+s)^{-\a-n/2}-1\right ]\,ds\nonumber\\
 &+&\frac{1}{\Gam (\a)} \intl_1^\infty  s^{\a-1} (1+s)^{-\a-n/2}\, ds-\frac{1}{\Gam (\a)} \intl_1^\infty  s^{\a-1} \, ds\nonumber\\
&=& \frac{1}{\Gam (\a)} \intl_0^\infty \frac{(1+s)^{-\a-n/2} -1}{s^{1-\a}}\, ds.\nonumber\eea
Thus, for $Re \, \a \in (-1,0)$, we have
\[
\frac{1}{\Gam (\a)} \intl_0^\infty \frac{(1+s)^{-\a-n/2} -1}{s^{1-\a}}\, ds=\frac{\Gam (n/2)}{\Gam (n/2+\a)}.\]
Replacing $\a$ by $-\a$, we obtain
\[\frac{1}{\Gam (-\a)} \intl_0^\infty \frac{(1+s)^{\a-n/2} -1}{s^{1+\a}}\, ds=\frac{\Gam (n/2)}{\Gam (n/2-\a)},\]
which gives  (\ref{ruak}).
\end{proof}

\bibliographystyle{amsplain}

\end{document}